\documentclass[11pt]{amsart}
\usepackage{amsmath}
\usepackage[active]{srcltx}
\usepackage{t1enc}
\usepackage[latin2]{inputenc}
\usepackage{verbatim}
\usepackage{amsmath,amsfonts,amssymb,amsthm}
\usepackage[mathcal]{eucal}
\usepackage{enumerate}
\usepackage[centertags]{amsmath}
\usepackage{graphics}

\newtheorem{theorem}{Theorem}

\newtheorem*{rodin}{Theorem R}
\newtheorem*{rodin2}{Theorem R2}
\newtheorem*{karagula}{Theorem K}
\newtheorem*{schipp}{Theorem Sch}
\newtheorem*{ggk1}{Theorem GGK1}
\newtheorem*{ggk2}{Theorem GGK2}
\newtheorem*{ggk3}{Theorem GGK3}
\newtheorem*{gog}{Theorem GG}

\begin{document}
\author{Ushangi Goginava}
\title[ Almost Everywhere Summability]{ Almost Everywhere Strong Summability
of two-dimensional Walsh-Fourier Series}
\address{U. Goginava, Department of Mathematics, Faculty of Exact and
Natural Sciences, Ivane Javakhishvili Tbilisi State University,
Chavcha\-vadze str. 1, Tbilisi 0128, Georgia}
\email{zazagoginava@gmail.com}
\date{}
\maketitle

\begin{abstract}
A BMO-estimation of two-dimensional Walsh-Fourier series is proved from
which an almost everywhere exponential summability of quadratic partial sums
of double Walsh-Fourier series is derived.
\end{abstract}

\medskip

\footnotetext{%
\par
Research was supported by project Shota Rustaveli National Science
Foundation grant DI/9/5-100/13 (Function spaces, weighted inequalities for
integral operators and problems of summability of Fourier series)}

\section{ Introduction}

We shall denote the set of all non-negative integers by $\mathbb{N}$ , the
set of all integers by$\,\,\mathbb{Z}$ and the set of dyadic rational
numbers in the unit interval $\mathbb{I}:=[0,1)$ by $\mathbb{Q}$. In
particular, each element of $\mathbb{Q}$ has the form $\frac{p}{2^{n}}$ for
some $p,n\in \mathbb{N},\,\,\,0\leq p<2^{n}$. Set $I_{N}:=[0,2^{-N}),$ $%
I_{N}\left( x\right) :=x\oplus I_{N}$, where $\oplus $ is the dyadic
addition (see \cite{SWS}).

Let $r_{0}\left( x\right) $ be the function defined by 
\begin{equation*}
r_{0}\left( x\right) =\left\{ 
\begin{array}{c}
1,\mbox{ if }x\in [0,1/2) \\ 
-1,\mbox{ if }x\in [1/2,1)%
\end{array}
\right. ,\,\,\,\,\,\,\,\,\,\,\,r_{0}\left( x+1\right) =r_{0}\left( x\right) .
\end{equation*}
The Rademacher system is defined by 
\begin{equation*}
r_{n}\left( x\right) =r_{0}\left( 2^{n}x\right) ,\,\,\,\,\,\,\,n\geq 1.
\end{equation*}

Let $w_{0},w_{1},...\,\,$represent the Walsh functions, i.e. $w_{0}\left(
x\right) =1\,\,$and $\,$if$\,\,k=2^{n_{1}}+\cdots +2^{n_{s}}\,$is a positive
integer with $n_{1}>n_{2}>\cdots >n_{s}\,\,$then 
\begin{equation*}
w_{k}\left( x\right) =r_{n_{1}}\left( x\right) \cdots r_{n_{s}}\left(
x\right) .
\end{equation*}

The Walsh-Dirichlet kernel is defined by 
\begin{equation*}
D_{n}\left( x\right) =\sum\limits_{k=0}^{n-1}w_{k}\left( x\right) ,n\in 
\mathbb{N}.
\end{equation*}

Given $x\in \mathbb{I}$, the expansion

\begin{equation}
x=\sum\limits_{k=0}^{\infty }x_{k}2^{-(k+1)},  \label{rep}
\end{equation}%
where each $x_{k}=0$ or $1$, will be called a dyadic expansion of $x.$ If $%
x\in \mathbb{I}\backslash \mathbb{Q}\mathbf{\,}$, then (\ref{rep}) is
uniquely determined. For the dyadic expansion $x\in \mathbb{Q}$\textbf{\ }we
choose the one for which $\lim\limits_{k\rightarrow \infty }x_{k}=0$.

The dyadic sum of $x,y\in \mathbb{I}$ in terms of the dyadic expansion of $x$
and $y$ is defined by 
\begin{equation*}
x\oplus y=\sum\limits_{k=0}^{\infty }\left\vert x_{k}-y_{k}\right\vert
2^{-(k+1)}.
\end{equation*}

We consider the double system $\left\{ w_{n}(x)\times w_{m}(y):\,n,m\in 
\mathbb{N}\right\} $ on the unit square $\mathbb{I}^{2}=\left[ 0,1\right)
\times \left[ 0,1\right) .$The notiation $a\lesssim b$ in the whole paper
stands for $a\leq c\cdot b$, where $c$ is an absolute constant.

The norm (or quasinorm) of the space $L_{p}\left( \mathbb{I}^{2}\right) $ is
defined by 
\begin{equation*}
\left\Vert f\right\Vert _{p}:=\left( \int\limits_{\mathbb{I}^{2}}\left\vert
f\left( x,y\right) \right\vert ^{p}dxdy\right) ^{1/p}\,\,\,\,\left(
0<p<+\infty \right) .
\end{equation*}

If $f\in L_{1}\left( \mathbb{I}^{2}\right) ,$ then 
\begin{equation*}
\hat{f}\left( n,m\right) =\int\limits_{I^{2}}f\left( x,y\right)
w_{n}(x)w_{m}(y)dxdy
\end{equation*}%
is the $\left( n,m\right) $-th Fourier coefficient of $f.$

The rectangular partial sums of double Fourier series with respect to the
Walsh system are defined by 
\begin{equation*}
S_{M,N}\left( x,y;f\right) =\sum\limits_{m=0}^{M-1}\sum\limits_{n=0}^{N-1}%
\hat{f}\left( m,n\right) w_{m}(x)w_{n}(y).
\end{equation*}

Denote 
\begin{equation*}
S_{n}^{\left( 1\right) }\left( x,y;f\right) :=\sum\limits_{l=0}^{n-1}%
\widehat{f}\left( l,y\right) w_{l}\left( x\right) ,
\end{equation*}%
\begin{equation*}
S_{m}^{\left( 2\right) }\left( x,y;f\right) :=\sum\limits_{r=0}^{m-1}%
\widehat{f}\left( x,r\right) w_{r}\left( y\right) ,
\end{equation*}%
where 
\begin{equation*}
\widehat{f}\left( l,y\right) =\int\limits_{\mathbb{I}}f\left( x,y\right)
w_{l}\left( x\right) dx
\end{equation*}%
and%
\begin{equation*}
\widehat{f}\left( x,r\right) =\int\limits_{\mathbb{I}}f\left( x,y\right)
w_{r}\left( y\right) dy
\end{equation*}

Recall the definition of $BMO\left[ \mathbb{I}\right] $ space. It is the
Banach space of functions $f\in L_{1}\left( \mathbb{I}\right) $ with the norm%
\begin{equation*}
\left\Vert f\right\Vert _{BMO}:=\sup\limits_{I}\left( \frac{1}{\left\vert
I\right\vert }\int\limits_{I}\left\vert f-f_{I}\right\vert ^{2}\right)
^{1/2}+\left\vert \int\limits_{\mathbb{I}}f\right\vert
\end{equation*}%
and the supremum is taken over all dyadic intervals $I\subset \mathbb{I}$.
Let $\xi :=\left\{ \xi _{n}:n=0,1,2,...\right\} $ be an arbitrary sequence
of numbers. Taking%
\begin{equation*}
\delta _{k}^{n}:=\left[ \frac{k}{2^{n}},\frac{k+1}{2^{n}}\right) ,
\end{equation*}%
we define%
\begin{equation*}
BMO\left[ \xi _{n}\right] :=\sup\limits_{0\leq n<\infty }\left\Vert
\sum\limits_{k=0}^{2^{n}-1}\xi _{k}\mathbb{I}_{\delta _{k}^{n}}\left(
t\right) \right\Vert _{BMO},
\end{equation*}%
where $\mathbb{I}_{E}$ is the characteristic function of $E\subset \mathbb{I}
$.

Set%
\begin{equation*}
F:=\left\{ J:=\left[ j2^{m},\left( j+1\right) 2^{m}\right) \cap \mathbb{N}%
,j,m\in \mathbb{N}\right\} .
\end{equation*}%
Then $F$ is the collection of integer dyadic intervals. The number of
elements in $J\in F$ will be denoted by $\left\vert J\right\vert $. The mean
value of the sequence $\xi :=\left\{ \xi _{n}:n=0,1,2,...\right\} $ with
respect to $J$ is defined by%
\begin{equation*}
\xi ^{J}:=\frac{1}{\left\vert J\right\vert }\sum\limits_{l\in J}\xi _{l}.
\end{equation*}%
Then it is easy to see that%
\begin{equation*}
BMO\left[ \xi _{n}\right] =\sup\limits_{J\in F}\left( \frac{1}{\left\vert
J\right\vert }\sum\limits_{k\in J}\left\vert \xi _{k}-\xi ^{J}\right\vert
^{2}\right) ^{1/2}.
\end{equation*}

We denote by $L\left( \log L\right) ^{\alpha }\left( \mathbb{I}^{2}\right) $
the class of measurable functions $f$, with $\int\limits_{\mathbb{I}%
^{2}}|f|\left( \log ^{+}|f|\right) ^{\alpha }<\infty ,$where $\log ^{+}u:=%
\mathbb{I}_{(1,\infty )}\left( u\right) \log u$.

Denote by $S_{n}^{T}(x,f)$ the partial sums of the trigonometric Fourier
series of $f$ and let 
\begin{equation*}
\sigma _{n}^{T}(x,f)=\frac{1}{n+1}\sum_{k=0}^{n}S_{k}^{T}(x,f)
\end{equation*}%
be the $(C,1)$ means. Fejér \cite{Fe} proved that $\sigma _{n}^{T}(f)$
converges to $f$ uniformly for any $2\pi $-periodic continuous function.
Lebesgue in \cite{Le} established almost everywhere convergence of $(C,1)$
means if $f\in L_{1}(\mathbb{T}),\mathbb{T}:=[-\pi ,\pi )$. The strong
summability problem, i.e. the convergence of the strong means 
\begin{equation}
\frac{1}{n+1}\sum\limits_{k=0}^{n}\left\vert S_{k}^{T}\left( x,f\right)
-f\left( x\right) \right\vert ^{p},\quad x\in \mathbb{T},\quad p>0,
\label{Hp}
\end{equation}%
was first considered by Hardy and Littlewood in \cite{H-L}. They showed that
for any $f\in L_{r}(\mathbb{T})~\left( 1<r<\infty \right) $ the strong means
tend to $0$ a.e., if $n\rightarrow \infty $. The Fourier series of $f\in
L_{1}(\mathbb{T})$ is said to be $\left( H,p\right) $-summable at $x\in T$,
if the values (\ref{Hp}) converge to $0$ as $n\rightarrow \infty $. The $%
\left( H,p\right) $-summability problem in $L_{1}(\mathbb{T})$ has been
investigated by Marcinkiewicz \cite{Ma} for $p=2$, and later by Zygmund \cite%
{Zy2} for the general case $1\leq p<\infty $. Oskolkov in \cite{Os} proved
the following: Let $f\in L_{1}(\mathbb{T})$ and let $\Phi $ be a continuous
positive convex function on $[0,+\infty )$ with $\Phi \left( 0\right) =0$
and 
\begin{equation}
\ln \Phi \left( t\right) =O\left( t/\ln \ln t\right) \text{ \ \ \ }\left(
t\rightarrow \infty \right) .  \label{a1}
\end{equation}%
Then for almost all $x$%
\begin{equation}
\lim\limits_{n\rightarrow \infty }\frac{1}{n+1}\sum\limits_{k=0}^{n}\Phi
\left( \left\vert S_{k}^{T}\left( x,f\right) -f\left( x\right) \right\vert
\right) =0.  \label{osk}
\end{equation}

It was noted in \cite{Os} that Totik announced the conjecture that (\ref{osk}%
) holds almost everywhere for any $f\in L_{1}(\mathbb{T})$, provided 
\begin{equation}
\ln \Phi \left( t\right) =O\left( t\right) \quad \left( t\rightarrow \infty
\right) .  \label{a2}
\end{equation}%
In \cite{Ro} Rodin proved

\begin{rodin2}
Let $f\in L_{1}(\mathbb{T})$. Then for any $A>0$ 
\begin{equation*}
\lim\limits_{n\rightarrow \infty }\frac{1}{n+1}\sum\limits_{k=0}^{n}\left(
\exp \left( A\left\vert S_{k}^{T}\left( x,f\right) -f\left( x\right)
\right\vert \right) -1\right) =0
\end{equation*}%
for a. e. $x\in \mathbb{T}$.
\end{rodin2}

Karagulyan \cite{Ka} proved that the following is true.

\begin{karagula}
Suppose that a continuous increasing function $\Phi :[0,\infty )\rightarrow
\lbrack 0,\infty ),\Phi \left( 0\right) =0$, satisfies the condition%
\begin{equation*}
\limsup_{t\rightarrow +\infty }\frac{\log \Phi \left( t\right) }{t}=\infty .
\end{equation*}%
Then there exists a function $f\in L_{1}(\mathbb{T})$ for which the relation%
\begin{equation*}
\limsup_{n\rightarrow \infty }\frac{1}{n+1}\sum\limits_{k=0}^{n}\Phi \left(
\left\vert S_{k}^{T}\left( x,f\right) \right\vert \right) =\infty
\end{equation*}%
holds everywhere on $\mathbb{T}$.
\end{karagula}

For quadratic partial sums of two-dimensional trigonometric Fourier series
Marcinkiewicz \cite{Ma2} has proved, that if $f\in L\log L\left( \mathbb{T}%
^{2}\right) $,$\mathbb{T}:=[-\pi ,\pi )^{2}$, then%
\begin{equation*}
\lim\limits_{n\rightarrow \infty }\frac{1}{n+1}\sum\limits_{k=0}^{n}\left(
S_{kk}^{T}\left( x,y,f\right) -f\left( x,y\right) \right) =0
\end{equation*}%
for a. e. $\left( x,y\right) \in \mathbb{T}^{2}$. ~Zhizhiashvili \cite{Zh}
improved this result showing that class $L\log L\left( \mathbb{T}^{2}\right) 
$ can be replaced by $L_{1}\left( \mathbb{T}^{2}\right) $.

From a result of ~Konyagin \cite{Kon} it follows that for every $\varepsilon
>0$ there exists a function $f\in L\log ^{1-\varepsilon }\left( \mathbb{T}%
^{2}\right) $ such that 
\begin{equation}
\lim\limits_{n\rightarrow \infty }\frac{1}{n+1}\sum\limits_{k=0}^{n}\left%
\vert S_{kk}^{T}\left( x,y,f\right) -f\left( x,y\right) \right\vert \neq 0%
\text{ \ \ for a. e. }\left( x,y\right) \in \mathbb{T}^{2}.  \label{str}
\end{equation}

\bigskip

These results show that in the one dimensional case we have the same maximal
spaces for $(C,1)$ summability and for $(C,1)$ strong summability. That is,
in both cases we have $L_{1}\left( \mathbb{T}\right) $. But, the situation
changes as we step further to the case of two dimensional functions. In
other words, the spaces of functions with almost everywhere summable
Marcinkiewicz and strong Marcinkiewicz means are different.

In \cite{CA2} a BMO-estimation of two-dimensional trigonometric Fourier
series is proved from which an almost everywhere exponential summability of
quadratic partial sums of double Fourier series is derived.

The results on strong summation and approximation of trigonometric Fourier
series have been extended for several other orthogonal systems. For
instance, concerning the Walsh system see Schipp \cite{Sch1,Sch2,Sch3},
Fridli and Schipp \cite{FS,FS2}, Leindler \cite{Le,Le1,Le2,Le3,Le4}, Totik 
\cite{To1,To2,To3}, Rodin \cite{Ro1}, Weisz \cite{We,We2}, Gabisonia \cite%
{Ga}.

The problems of summability of cubical partial sums of multiple Fourier
series have been investigated by Gogoladze \cite{Gog1,Gog2,Gog3}, Wang \cite%
{Wa}, Zhag \cite{ZhHe}, Glukhov \cite{Gl}, Goginava \cite{Gogi}, G\'at,
Goginava, Tkebuchava \cite{GGT}, Goginava, Gogoladze \cite{GG} .

For Walsh system Rodin \cite{rodin} (see also Schipp \cite{Sch}) proved that
the following is true.

\begin{rodin}[Rodin]
If $\Phi (t):[0,\infty )\rightarrow \lbrack 0,\infty )$, $\Phi (0)=0$, is an
increasing continuous function satisfying 
\begin{equation}
\limsup_{t\rightarrow \infty }\frac{\log \Phi (t)}{t}<\infty ,  \label{rod}
\end{equation}%
then the partial sums of Walsh-Fourier series of any function $f\in
L^{1}\left( \mathbb{I}\right) $ satisfy the condition%
\begin{equation*}
\lim\limits_{n\rightarrow \infty }\frac{1}{n}\sum\limits_{k=1}^{n}\Phi
\left( \left\vert S_{k}\left( x;f\right) -f\left( x\right) \right\vert
\right) =0
\end{equation*}%
almost everywhere on $\mathbb{I}$.
\end{rodin}

In the paper \cite{JMAA} we established, that, as in trigonometric case \cite%
{Ka}, the bound (\ref{rod}) is sharp for a.e. $\Phi $-summability of
Walsh-Fourier series. Moreover, we prove

\begin{ggk1}
If an increasing function $\Phi (t):[0,\infty )\rightarrow \lbrack 0,\infty
) $ satisfies the condition%
\begin{equation*}
\limsup_{t\rightarrow \infty }\frac{\log \Phi (t)}{t}=\infty ,
\end{equation*}
then there exists a function $f\in L^{1}\left( \mathbb{I}\right) $ such that 
\begin{equation*}
\limsup_{n\rightarrow \infty }\frac{1}{n}\sum\limits_{k=1}^{n}\Phi \left(
\left\vert S_{k}\left( x;f\right) \right\vert \right) =\infty
\end{equation*}%
holds everywhere on $[0,1)$.
\end{ggk1}

Schipp in \cite{Sch} introduced the following operator%
\begin{equation*}
V_{n}\left( x;f\right) :=\left( \frac{1}{2^{n}}\int\limits_{\mathbb{I}%
}\left( \sum\limits_{j=0}^{n-1}2^{j-1}\mathbb{I}_{I_{j}}\left( t\right)
S_{2^{n}}f\left( x\oplus t\oplus e_{j}\right) \right) ^{2}dt\right) ^{1/2}.
\end{equation*}

Let%
\begin{equation*}
V\left( f\right) :=\sup\limits_{n}V_{n}\left( f\right) .
\end{equation*}

The following theorem is proved by Schipp.

\begin{schipp}[\protect\cite{Sch}]
\label{SchD} Let $f\in L_{1}\left( \mathbb{I}\right) $. Then%
\begin{equation*}
\mu \left\{ \left\vert Vf\right\vert >\lambda \right\} \lesssim \frac{%
\left\Vert f\right\Vert _{1}}{\lambda }.
\end{equation*}
\end{schipp}

\bigskip Set%
\begin{equation*}
H_{n}^{p}f:=\left( \frac{1}{2^{n}}\sum\limits_{m=0}^{2^{n}-1}\left\vert
S_{mm}f\right\vert ^{p}\right) ^{1/p}
\end{equation*}%
and the maximal strong operator%
\begin{equation*}
H_{\ast }^{p}f:=\sup\limits_{n\in \mathbb{N}}H_{n}^{p}f,\text{ \ \ \ }p>0.
\end{equation*}

In \cite{AM} we studied the a. e. convergence of strong Marcinkiewicz means
of the two-dimensional Walsh-Fourier series. In particular, the following is
true.

\begin{ggk2}
Let $f\in L\log L\left( \mathbb{I}^{2}\right) $ and $p>0$. Then%
\begin{equation*}
\mu \left\{ H_{\ast }^{p}f>\lambda \right\} \lesssim \frac{1}{\lambda }%
\left( 1+\int\limits_{\mathbb{I}^{2}}\left\vert f\right\vert \log
^{+}\left\vert f\right\vert \right) .
\end{equation*}
\end{ggk2}

The weak type $\left( L\log ^{+}L,1\right) $ inequality and the usual
density argument of Marcinkiewicz and Zygmund imply

\begin{ggk3}
Let $f\in L\log L\left( \mathbb{I}^{2}\right) $ and $p>0$. Then 
\begin{equation*}
\left( \frac{1}{n}\sum\limits_{m=0}^{n-1}\left\vert S_{mm}\left(
x,y,f\right) -f\left( x,y\right) \right\vert ^{p}\right) ^{1/p}\rightarrow 0%
\text{ for a.e. }\left( x,y\right) \in \mathbb{I}^{2}\text{ as }n\rightarrow
\infty .
\end{equation*}
\end{ggk3}

We note that from the theorem of Getsadze \cite{Gets} it follows that the
class $L\log L$ in the last theorem is necessary in the context of strong
summability question. That is, it is not possible to give a larger
convergence space (of the form $L\log L\phi (L)$ with $\phi (\infty )=0$)
than $L\log L$. This means a sharp contrast between the one and two
dimensional strong summability.

In \cite{GG} the exponential uniform strong approximation of the
Marcinkiewicz means of the two-dimensional Walsh-Fourier series was studied.
We say that the function $\psi $ belongs to the class $\Psi $ if it increase
on $[0,+\infty )$ and 
\begin{equation*}
\lim\limits_{u\rightarrow 0}\psi \left( u\right) =\psi \left( 0\right) =0.
\end{equation*}

\begin{gog}[\protect\cite{GG}]
a)Let $\varphi \in \Psi $ and let the inequality 
\begin{equation*}
\overline{\lim\limits_{u\rightarrow \infty }}\frac{\varphi \left( u\right) }{%
\sqrt{u}}<\infty
\end{equation*}%
hold. Then for any function $f\in C\left( \mathbb{I}^{2}\right) $ the
equality 
\begin{equation*}
\lim\limits_{n\rightarrow \infty }\left\Vert \frac{1}{n}\sum%
\limits_{l=1}^{n}\left( e^{\varphi \left( \left\vert S_{ll}\left( f\right)
-f\right\vert \right) }-1\right) \right\Vert _{C}=0
\end{equation*}%
is satisfied.

b) For any function $\varphi \in \Psi $ satisfying the condition 
\begin{equation*}
\overline{\lim\limits_{u\rightarrow \infty }}\frac{\varphi \left( u\right) }{%
\sqrt{u}}=\infty
\end{equation*}%
there exists a function $F\in C\left( \mathbb{I}^{2}\right) $ such that 
\begin{equation*}
\overline{\lim\limits_{m\rightarrow \infty }}\frac{1}{m}\sum%
\limits_{l=1}^{m}\left( e^{\varphi \left( \left\vert S_{ll}\left(
0,0,F\right) -F\left( 0,0\right) \right\vert \right) }-1\right) =+\infty .
\end{equation*}
\end{gog}

In this paper we study a BMO-estimation for quadratic partial sums of
two-dimensional Walsh-Fourier series from which an almost everywhere
exponential summability of quadratic partial sums of double Walsh-Fourier
series is derived.

\begin{theorem}
\label{bmo}If $f\in L\left( \log L\right) ^{2}\left( \mathbb{I}^{2}\right) $%
, then%
\begin{equation*}
\mu \left\{ \left( x,y\right) \in I^{2}:\text{BMO}\left[ S_{nn}\left(
x,y;f\right) \right] >\lambda \right\} \lesssim \frac{1}{\lambda }\left(
1+\int\limits_{\mathbb{I}^{2}}|f|\left( \log |f|\right) ^{2}\right) .
\end{equation*}
\end{theorem}

The following theorem shows that the quadratic sums of two-dimensional
Walsh-Fourier series of a function $f\in L\left( \log L\right) ^{2}\left( 
\mathbb{I}^{2}\right) $ are almost everywhere exponentially summable to the
function $f$. It will be obtained from the previous theorem by using the
John-Nirenberg theorem (see (\cite{CA2}).

\begin{theorem}
\label{a.e.exp} Suppose that $f\in L\left( \log L\right) ^{2}\left( \mathbb{I%
}^{2}\right) $. Then for any $A>0$%
\begin{equation*}
\lim\limits_{m\rightarrow \infty }\frac{1}{m}\sum\limits_{n=1}^{m}\left(
\exp \left( A\left\vert S_{nn}\left( x,y;f\right) -f\left( x,y\right)
\right\vert \right) -1\right) =0
\end{equation*}%
for a. e. $\left( x,y\right) \in \mathbb{I}^{2}$.
\end{theorem}

\section{Proof of Theorem}

Let $f\in L_{1}\left( \mathbb{I}^{2}\right) $. Then the dyadic maximal
function is given by 
\begin{equation*}
Mf\left( x,y\right) :=\sup\limits_{n\in \mathbb{N}}2^{2n}\int\limits_{I_{n}%
\left( x\right) \times I_{n}\left( y\right) }\left\vert f\left( s,t\right)
\right\vert dsdt.\,\,
\end{equation*}

For a two-dimensional integrable function $f$ we need to introduce the
following hybrid maximal functions 
\begin{equation*}
M_{1}f\left( x,y\right) :=\sup\limits_{n\in \mathbb{N}}2^{n}\int%
\limits_{I_{n}\left( x\right) }\left\vert f\left( s,y\right) \right\vert ds,
\end{equation*}%
\begin{equation*}
M_{2}f\left( x,y\right) :=\sup\limits_{n\in \mathbb{N}}2^{n}\int%
\limits_{I_{n}\left( y\right) }\left\vert f\left( x,t\right) \right\vert dt,
\end{equation*}%
\begin{eqnarray}
&&V_{1}\left( x,y,f\right)  \label{V1} \\
&:&=\sup\limits_{n\in \mathbb{N}}\left( \frac{1}{2^{n}}\int\limits_{\mathbb{I%
}}\left( \sum\limits_{j=0}^{n-1}2^{j-1}\mathbb{I}_{I_{j}}\left( t\right)
S_{2^{n}}^{(1)}f\left( x\oplus t\oplus e_{j},y\right) \right) ^{2}dt\right)
^{1/2},  \notag
\end{eqnarray}%
\begin{eqnarray}
&&V_{2}\left( x,y,f\right)  \label{V2} \\
&:&=\sup\limits_{n\in \mathbb{N}}\left( \frac{1}{2^{n}}\int\limits_{\mathbb{I%
}}\left( \sum\limits_{j=0}^{n-1}2^{j-1}\mathbb{I}_{I_{j}}\left( t\right)
S_{2^{n}}^{(2)}f\left( x,y\oplus t\oplus e_{j}\right) \right) ^{2}dt\right)
^{1/2}.  \notag
\end{eqnarray}%
It is well known \cite{Zy2} that for $f\in L\log ^{+}L$ the following
estimation holds%
\begin{equation}
\lambda \mu \left\{ Mf>\lambda \right\} \lesssim 1+\int\limits_{\mathbb{I}%
^{2}}\left\vert f\right\vert \log ^{+}\left\vert f\right\vert  \label{max1}
\end{equation}%
and $\ $for $s=1,2$%
\begin{equation}
\int\limits_{\mathbb{I}^{2}}M_{s}f\lesssim 1+\int\limits_{\mathbb{I}%
^{2}}\left\vert f\right\vert \log ^{+}\left\vert f\right\vert ,  \label{max2}
\end{equation}

\begin{equation}
\mu \left\{ :V_{s}\left( f\right) >\lambda \right\} \lesssim \frac{%
\left\Vert f\right\Vert _{1}}{\lambda },\text{ \ \ }f\in L_{1}\left( \mathbb{%
I}^{2}\right) .  \label{weakV2}
\end{equation}

It is proved in \cite{AM} that the following estimation holds%
\begin{eqnarray}
&&\left( \frac{1}{2^{n}}\sum\limits_{m=0}^{2^{n}-1}\left\vert S_{mm}\left(
x,y,f\right) \right\vert ^{2}\right) ^{1/2}  \label{mainest} \\
&\lesssim &V_{2}\left( x,y,M_{1}f\right) +V_{1}\left( x,y,M_{2}f\right)
+Mf\left( x,y\right)  \notag \\
&&+V_{2}\left( x,y,A\right) +V_{1}\left( x,y,A\right) +\left\Vert
f\right\Vert _{1},  \notag
\end{eqnarray}%
where $A$ is an integrable and nonegative function on $\mathbb{I}^{2}$ of
two variable for which

\begin{equation}
\int\limits_{\mathbb{I}^{2}}A\lesssim 1+\int\limits_{\mathbb{I}%
^{2}}\left\vert f\right\vert \log ^{+}\left\vert f\right\vert ,\text{ \ \ }%
f\in L\log L.  \label{LlogL}
\end{equation}

\begin{proof}[Proof of Theorem \protect\ref{bmo}]
We can wtite%
\begin{eqnarray}
&&BMO\left[ S_{nn}\left( x,y;f\right) \right]  \label{bmoS} \\
&=&\sup\limits_{m,j}\left( \frac{1}{2^{m}}\sum\limits_{l=j2^{m}}^{\left(
j+1\right) 2^{m}-1}\left\vert S_{ll}\left( x,y;f\right) -\frac{1}{2^{m}}%
\sum\limits_{q=j2^{m}}^{\left( j+1\right) 2^{m}-1}S_{qq}\left( x,y;f\right)
\right\vert ^{2}\right) ^{1/2}  \notag \\
&=&\sup\limits_{m,j}\left( \frac{1}{2^{m}}\sum\limits_{l=0}^{2^{m}-1}\left%
\vert S_{l+j2^{m},l+j2^{m}}\left( x,y;f\right) \right. \right.  \notag \\
&&\left. \left. -\frac{1}{2^{m}}\sum%
\limits_{q=0}^{2^{m}-1}S_{q+j2^{m},q+j2^{m}}\left( x,y;f\right) \right\vert
^{2}\right) ^{1/2}.  \notag
\end{eqnarray}%
Since $\left( 0\leq l<2^{m}\right) $%
\begin{eqnarray*}
S_{l+j2^{m},l+j2^{m}}\left( x,y;f\right) &=&S_{j2^{m},j2^{m}}\left(
x,y;f\right) +S_{j2^{m},l}\left( x,y;fw_{j2^{m}}\left( x\right) \right)
w_{j2^{m}}\left( y\right) \\
&&+S_{l,j2^{m}}\left( x,y;fw_{j2^{m}}\left( y\right) \right)
w_{j2^{m}}\left( x\right) \\
&&+S_{l,l}\left( x,y;fw_{j2^{m}}\left( x\right) \otimes w_{j2^{m}}\left(
y\right) \right) w_{j2^{m}}\left( x\right) w_{j2^{m}}\left( y\right)
\end{eqnarray*}%
from (\ref{bmoS}) we obtain%
\begin{eqnarray}
&&BMO\left[ S_{nn}\left( x,y;f\right) \right]  \label{T1-T3} \\
&\leq &\sup\limits_{m,j}\left( \frac{1}{2^{m}}\sum\limits_{l=0}^{2^{m}-1}%
\left\vert S_{l,l}\left( x,y;fw_{j2^{m}}\left( x\right) \otimes
w_{j2^{m}}\left( y\right) \right) \right. \right.  \notag \\
&&\left. \left. -\frac{1}{2^{m}}\sum\limits_{q=0}^{2^{m}-1}S_{q,q}\left(
x,y;fw_{j2^{m}}\left( x\right) \otimes w_{j2^{m}}\left( y\right) \right)
\right\vert ^{2}\right) ^{1/2}  \notag
\end{eqnarray}%
\begin{eqnarray*}
&&+\sup\limits_{m,j}\left( \frac{1}{2^{m}}\sum\limits_{l=0}^{2^{m}-1}\left%
\vert S_{l,j2^{m}}\left( x,y;fw_{j2^{m}}\left( y\right) \right) \right.
\right. \\
&&\left. \left. -\frac{1}{2^{m}}\sum\limits_{q=0}^{2^{m}-1}S_{q,j2^{m}}%
\left( x,y;fw_{j2^{m}}\left( y\right) \right) \right\vert ^{2}\right) ^{1/2}
\end{eqnarray*}%
\begin{eqnarray*}
&&+\sup\limits_{m,j}\left( \frac{1}{2^{m}}\sum\limits_{l=0}^{2^{m}-1}\left%
\vert S_{j2^{m},l}\left( x,y;fw_{j2^{m}}\left( x\right) \right) \right.
\right. \\
&&\left. \left. -\frac{1}{2^{m}}\sum\limits_{q=0}^{2^{m}-1}S_{j2^{m},q}%
\left( x,y;fw_{j2^{m}}\left( x\right) \right) \right\vert ^{2}\right) ^{1/2}
\end{eqnarray*}%
\begin{eqnarray*}
&\leq &2\sup\limits_{m,j}\left( \frac{1}{2^{m}}\sum\limits_{l=0}^{2^{m}-1}%
\left\vert S_{l,l}\left( x,y;fw_{j2^{m}}\left( x\right) \otimes
w_{j2^{m}}\left( y\right) \right) \right\vert ^{2}\right) ^{1/2} \\
&&+2\sup\limits_{m,j}\left( \frac{1}{2^{m}}\sum\limits_{l=0}^{2^{m}-1}\left%
\vert S_{l,j2^{m}}\left( x,y;fw_{j2^{m}}\left( y\right) \right) \right\vert
^{2}\right) ^{1/2} \\
&&+2\sup\limits_{m,j}\left( \frac{1}{2^{m}}\sum\limits_{l=0}^{2^{m}-1}\left%
\vert S_{j2^{m},l}\left( x,y;fw_{j2^{m}}\left( x\right) \right) \right\vert
^{2}\right) ^{1/2} \\
&:&=T_{1}+T_{2}+T_{3}.
\end{eqnarray*}

From (\ref{mainest}) we have%
\begin{eqnarray}
T_{1} &\lesssim &V_{2}\left( x,y,M_{1}f\right) +V_{1}\left(
x,y,M_{2}f\right) +Mf\left( x,y\right)  \label{T1} \\
&&+V_{2}\left( x,y,A\right) +V_{1}\left( x,y,A\right) +\left\Vert
f\right\Vert _{1}.  \notag
\end{eqnarray}

Since%
\begin{equation*}
S_{l,j2^{m}}\left( x,y;fw_{j2^{m}}\right) =S_{l}^{\left( 1\right) }\left(
x,y;S_{j2^{m}}^{\left( 2\right) }w_{j2^{m}}\right)
\end{equation*}%
for $T_{2}$ we can write%
\begin{equation}
T_{2}\lesssim \sup\limits_{m,j}\left( \frac{1}{2^{m}}\sum%
\limits_{l=0}^{2^{m}-1}\left\vert S_{l}^{\left( 1\right) }\left(
x,y;S_{j2^{m}}^{\left( 2\right) }\left( fw_{j2^{m}}\right) \right)
\right\vert ^{2}\right) ^{1/2}.  \label{T2-1}
\end{equation}

Schipp proved the following estimation (see \cite{Sch})%
\begin{equation}
\left( \frac{1}{2^{m}}\sum\limits_{l=0}^{2^{m}-1}\left\vert S_{l}\left(
x;f\right) \right\vert ^{2}\right) ^{1/2}\lesssim V\left( x,f\right) .
\label{Sch}
\end{equation}

Combining (\ref{T2-1}) and (\ref{Sch}) we get%
\begin{equation*}
T_{2}\lesssim \sup\limits_{m,j}V_{1}\left( x,y;\left\vert S_{j2^{m}}^{\left(
2\right) }\left( f\right) \right\vert \right) \lesssim V_{1}\left(
x,y;S_{\ast }^{\left( 2\right) }\left( f\right) \right) ,
\end{equation*}%
where%
\begin{equation*}
S_{\ast }^{\left( 2\right) }\left( x,y;f\right) :=\sup\limits_{n}\left\vert
S_{n}^{\left( 2\right) }\left( x,y;f\right) \right\vert .
\end{equation*}

Let $f\in L\left( \log L\right) ^{2}\left( \mathbb{I}^{2}\right) $. Then $%
f\left( x,\cdot \right) \in L\left( \log L\right) ^{2}\left( \mathbb{I}%
\right) $ for a. e. $x\in \mathbb{I},$and from the well-known theorem  (\cite%
{tateoka}) $S_{\ast }^{\left( 2\right) }\left( x,\cdot ;f\right) \in
L_{1}\left( \mathbb{I}\right) $ for a. e. $x\in \mathbb{I}$. Moreover,%
\begin{equation}
\int\limits_{\mathbb{I}}\left\vert S_{\ast }^{\left( 2\right) }\left(
x,y;f\right) \right\vert dy\lesssim \left( \int\limits_{\mathbb{I}%
}\left\vert f\left( x,y\right) \right\vert \left( \log ^{+}\left\vert
f\left( x,y\right) \right\vert \right) ^{2}dy+1\right)   \label{LlogL2}
\end{equation}%
for a. e. $x\in \mathbb{I}$.

Setting

\begin{equation*}
\Omega :=\left\{ \left( x,y\right) \in \mathbb{I}^{2}:V_{1}\left(
x,y,f\right) >\lambda \right\} .
\end{equation*}%
we can use Fubini's Theorem and Theorem \ref{SchD} to write%
\begin{eqnarray}
\left\vert \Omega \right\vert  &=&\int\limits_{\mathbb{I}^{2}}1_{\Omega
}\left( x,y\right) dxdy  \label{weakV1} \\
&=&\int\limits_{\mathbb{I}}\left( \int\limits_{\mathbb{I}}1_{\Omega }\left(
x,y\right) dx\right) dy  \notag \\
&\lesssim &\frac{1}{\lambda }\int\limits_{\mathbb{I}}\left( \int\limits_{%
\mathbb{I}}\left\vert f\left( x,y\right) \right\vert dx\right) dy.  \notag
\end{eqnarray}%
Consequently, from (\ref{LlogL2}) we obtain%
\begin{eqnarray}
&&\left\vert \left\{ \left( x,y\right) \in \mathbb{I}^{2}:V_{1}\left(
x,y;S_{\ast }^{\left( 2\right) }\left( f\right) \right) >\lambda \right\}
\right\vert   \label{T2} \\
&\lesssim &\frac{1}{\lambda }\int\limits_{\mathbb{I}}\left( \int\limits_{%
\mathbb{I}}\left\vert S_{\ast }^{\left( 2\right) }\left( x,y;f\right)
\right\vert dx\right) dy  \notag \\
&\lesssim &\frac{1}{\lambda }\int\limits_{\mathbb{I}}\left( \int\limits_{%
\mathbb{I}}\left\vert f\left( x,y\right) \right\vert \left( \log
^{+}\left\vert f\left( x,y\right) \right\vert \right) ^{2}dy+1\right) dx 
\notag \\
&=&\frac{c}{\lambda }\left( \int\limits_{\mathbb{I}^{2}}\left( \left\vert
f\left( x,y\right) \right\vert \left( \log ^{+}\left\vert f\left( x,y\right)
\right\vert \right) ^{2}+1\right) dxdy\right) ^{1/2}.  \notag
\end{eqnarray}

Analogously, we can prove that%
\begin{equation}
\left\vert \left\{ T_{3}>\lambda \right\} \right\vert \lesssim \frac{1}{%
\lambda }\left( \int\limits_{\mathbb{I}^{2}}\left( \left\vert f\left(
x,y\right) \right\vert \left( \log ^{+}\left\vert f\left( x,y\right)
\right\vert \right) ^{2}+1\right) dxdy\right) .  \label{T3}
\end{equation}

From (\ref{max1}), (\ref{max2}), (\ref{weakV1}), (\ref{weakV2}), (\ref%
{mainest}), (\ref{LlogL}) and Theorem D we conclude that%
\begin{eqnarray*}
&&\left\vert \left\{ T_{1}>\lambda \right\} \right\vert \\
&\lesssim &\frac{1}{\lambda }\left( \left\Vert M_{1}f\right\Vert
_{1}+\left\Vert M_{2}f\right\Vert _{1}+\left\Vert A\right\Vert
_{1}+\left\Vert f\right\Vert _{1}\right) \\
&\lesssim &\frac{1}{\lambda }\left( 1+\int\limits_{\mathbb{I}^{2}}\left\vert
f\right\vert \log ^{+}\left\vert f\right\vert \right) .
\end{eqnarray*}%
Combining (\ref{T1-T3}), (\ref{T1}), (\ref{T2}) and (\ref{T3}) we conclude
the proof of Theorem \ref{bmo}.
\end{proof}

\end{document}